\documentclass[a4, 12pt]{amsart}
\usepackage{amssymb}
\usepackage{amstext}
\usepackage{amsmath}
\usepackage{amscd}
\usepackage{latexsym}
\usepackage{amsfonts}
\usepackage[dvips]{graphicx}
\usepackage[all]{xy}
\usepackage{enumerate}
\usepackage{color}

\newtheorem{theorem}{Theorem}[section]
\newtheorem{proposition}[theorem]{Proposition}
\newtheorem{lemma}[theorem]{Lemma}
\newtheorem{corollary}[theorem]{Corollary}

\theoremstyle{definition}
\newtheorem{definition}[theorem]{Definition}

\theoremstyle{remark}
\newtheorem*{ac}{Acknowledgments}

\numberwithin{equation}{section}

\newtheorem*{tpf}{{\it Proof of Theorem \ref{thm0}}}

\def\mod{\mathrm{mod}}
\def\Coker{\mathrm{Cok}}

\def\rank{\mathrm{rank}}

\def\NF{\mathcal{V}}

\def\m{\mathfrak m}

\def\p{\mathfrak p}
\def\q{\mathfrak q}
\def\fP{\mathfrak P}

\def\Z{\Bbb Z}

\def\Spec{\mathrm{Spec}}

\def\CM{\mathrm{CM}}
\def\MF{\mathrm{MF}}

\begin{document}

\title[CM modules over hypersurfaces of countable CM type]{On the structure of Cohen-Macaulay modules over hypersurfaces of countable Cohen-Macaulay representation type}

\author{Tokuji Araya}
\address{Liberal Arts Division, Tokuyama College of Technology, Gakuendai, Shunan, Yamaguchi 745-8585, Japan}
\email{araya@tokuyama.ac.jp}

\author{Kei-ichiro Iima}
\address{Department of Liberal Studies, Nara National College of Technology, 22 Yata-cho, Yamatokoriyama, Nara 639-1080, Japan}
\email{iima@libe.nara-k.ac.jp}

\author{Ryo Takahashi}
\address{Department of Mathematical Sciences, Faculty of Science, Shinshu University, 3-1-1 Asahi, Matsumoto, Nagano 390-8621, Japan/Department of Mathematics, University of Nebraska, Lincoln, NE 68588-0130, USA}
\email{takahasi@math.shinshu-u.ac.jp}
\urladdr{http://math.shinshu-u.ac.jp/~takahasi}
\thanks{The third author was partially supported by JSPS Grant-in-Aid for Young Scientists (B) 22740008 and by JSPS Postdoctoral Fellowships for Research Abroad}
\begin{abstract}
Let $R$ be a complete local hypersurface over an algebraically closed field of characteristic different from two, and suppose that $R$ has countable Cohen-Macaulay (CM) representation type.
In this paper, it is proved that the maximal Cohen-Macaulay (MCM) $R$-modules which are locally free on the punctured spectrum are dominated by the MCM $R$-modules which are not locally free on the punctured spectrum.
More precisely, there exists a single $R$-module $X$ such that the indecomposable MCM $R$-modules not locally free on the punctured spectrum are $X$ and its syzygy $\Omega X$ and that any other MCM $R$-modules are obtained from extensions of $X$ and $\Omega X$.
\end{abstract}
\maketitle
\section{Introduction} 

Let $R$ be a complete local hypersurface over an algebraically closed field.
Suppose that the characteristic of $k$ is zero and that $R$ has finite Cohen-Macaulay (CM) representation type, namely, there exist only finitely many isomorphism classes of indecomposable maximal Cohen-Macaulay (MCM) $R$-modules.
Then $R$ is isomorphic to the residue ring $k[[x_0,x_1,x_2,\ldots ,x_d]]/(f)$ where $f$ is one of the following polynomials:
$$
\begin{array}{ll}
(A_n) (n\ge 1) & x_0^2+x_1^{n+1}+x_2^2+\cdots +x_d^2, \\
(D_n) (n\ge 4) & x_0^2x_1+x_1^{n-1}+x_2^2+\cdots +x_d^2, \\
(E_6) & x_0^3+x_1^4+x_2^2+\cdots +x_d^2, \\
(E_7) & x_0^3+x_0x_1^3+x_2^2+\cdots +x_d^2, \\
(E_8) & x_0^3+x_1^5+x_2^2+\cdots +x_d^2.
\end{array}
$$
In this case, all objects and morphisms in the category $\CM(R)$ of MCM $R$-modules have been classified completely, that is, the Auslander-Reiten quiver of the stable category $\underline{\CM}(R)$ of $\CM(R)$ has been obtained.
For the details, see \cite{BGS,K,Y}. 

Now, assume that $k$ has characteristic different from two, and that $R$ has countable CM representation type, namely, there exist infinitely but only countably many isomorphism classes of indecomposable MCM $R$-modules.
Then $R$ is isomorphic to $k[[x_0,x_1,x_2,\ldots ,x_d]]/(f)$ where $f$ is either of the following \cite[(1.6)]{GK}:
$$
\begin{array}{ll}
(A_\infty^d) & x_0^2+x_2^2+\cdots +x_d^2, \\
(D_\infty^d) & x_0^2x_1+x_2^2+\cdots +x_d^2.
\end{array}
$$
In this case, all objects in $\CM(R)$ have been classified completely \cite{BGS,BD,K}, but morphisms in $\CM(R)$ have not. 
The purpose of this paper is to investigate the relationships among objects in $\CM(R)$ by focusing on the objects that are locally free on the punctured spectrum of $R$.
Modules locally free on the punctured spectrum have recently been studied in relation to whose nonfree loci; see \cite{res,stcm}.
We also use nonfree loci to get our results.

Let us introduce here some notation.
We denote by $\mathcal{P}(R)$ the full subcategory of $\CM(R)$ consisting of all modules that are locally free on the punctured spectrum of $R$.
Let $\mathcal{M}(R)$ be the set of nonisomorphic indecomposable MCM $R$-modules that are {\em not} locally free on the punctured spectrum of $R$, and let $\NF(M)$ be the nonfree locus of a finitely generated $R$-module $M$.
(Recall that the nonfree locus of $M$ is defined as the set of prime ideals $\p$ of $R$ such that the $R_\p$-module $M_\p$ is nonfree.)
Let $F_1\overset{\partial}{\to}F_0\to M\to 0$ be part of a minimal free resolution of $M$.
Then the (first) syzygy of $M$ is by definition the image of the map $\partial$ and denoted by $\Omega M$.
(Note that it is uniquely determined up to isomorphism.)

The main theorem of this paper is the following. 

\begin{theorem}\label{thm0}
Let $k$ be an algebraically closed field of characteristic different from $2$.
Let $R$ be a complete local hypersurface over k of countable CM representation type.
Then, as we stated above, $R$ is (isomorphic to) a residue ring $k[[x_0,x_1,x_2,\ldots ,x_d]]/(f)$, where $f$ is either $({A}^d_{\infty})$ or $({D}^d_{\infty})$.
Let $\p_R = (x_0, x_2, \ldots , x_d)$ and $\m_R = (x_0, x_1, x_2, \ldots , x_d)$ be ideals of $R$. 
The following hold. 
\begin{enumerate}[\rm (1)]
\item
There exists an $R$-module $X_R$ such that
\begin{enumerate}[\rm (a)]
\item
$\mathcal{M}(R) = \{ X_R, \Omega(X_R)\}$,
\item
$\NF(X_R) = \{ \p_R, \m_R\} = \NF(\Omega(X_R))$.
\end{enumerate}
\item
For each indecomposable $R$-module $M \in \mathcal{P}(R)$, there is an exact sequence
$$
0 \to L \to M\oplus R^n \to N \to 0
$$
of $R$-modules with $L, N \in \mathcal{M}(R)$ and $n\ge 0$. 
\end{enumerate}
\end{theorem}

The first assertion of this theorem especially says that there exist at most two indecomposable MCM $R$-modules which are {\em not} locally free on the punctured spectrum of $R$.
The second assertion especially says that the MCM $R$-modules which are locally free on the punctured spectrum of $R$ are dominated by the MCM $R$-modules which are {\em not} locally free on the punctured spectrum of $R$.
On the other hand, the structure of the MCM modules locally free on the punctured spectrum has been clarified by Schreyer \cite{Sc}:

\begin{theorem}[Schreyer]
Let $n$ be a positive integer.
The Auslander-Reiten quiver of the stable category of $\mathcal{P}(R)$ has the following form.
\begin{enumerate}[\rm (1)]
\item
When $f=(A^{2n-1}_\infty)$:

\begin{picture}(400,40)
\put (54,25){$
\xymatrix{
\bullet \ar@{.}@(dr,dl) \ar@<.5ex>[r] & \ar@<.5ex>[l] \bullet \ar@{.}@(dr,dl) \ar@<.5ex>[r] & \ar@<.5ex>[l] \bullet \ar@{.}@(dr,dl) \ar@<.5ex>[r] & \ar@<.5ex>[l] \bullet \ar@{.}@(dr,dl) \ar@<.5ex>[r] & \ar@<.5ex>[l] \bullet \ar@{.}@(dr,dl) \ar@<.5ex>[r] & \ar@<.5ex>[l] \bullet \ar@{.}@(dr,dl) \ar@<.5ex>[r] & \ar@<.5ex>[l] \bullet \ar@{.}@(dr,dl) \ar@<.5ex>[r] & \ar@<.5ex>[l] \cdots
}
$}
\end{picture}
\item
When $f=(D^{2n-1}_\infty)$:

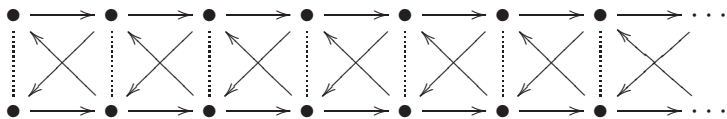
\begin{picture}(400,63)
\put (54,48){$
\xymatrix{
\bullet \ar@{.}[d] \ar[r] & \ar[ld] \bullet \ar@{.}[d] \ar[r] & \ar[ld] \bullet \ar@{.}[d] \ar[r] & \ar[ld] \bullet \ar@{.}[d] \ar[r] & \ar[ld] \bullet \ar@{.}[d] \ar[r] & \ar[ld] \bullet \ar@{.}[d] \ar[r] & \ar[ld] \bullet \ar@{.}[d] \ar[r] & \ar[ld] \cdots \\
\bullet \ar[r] & \ar[lu] \bullet \ar[r] & \ar[lu] \bullet \ar[r] & \ar[lu] \bullet \ar[r] & \ar[lu] \bullet \ar[r] & \ar[lu] \bullet \ar[r] & \ar[lu] \bullet \ar[r] & \ar[lu] \cdots
}
$}
\end{picture}
\item
When $f=(A^{2n}_\infty)$:

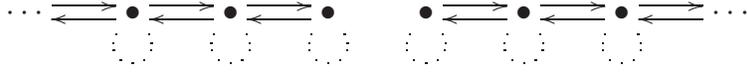
\begin{picture}(400,40)
\put (54,25){$
\xymatrix{
\cdots \ar@<.5ex>[r] & \ar@<.5ex>[l] \bullet \ar@{.}@(dr,dl) \ar@<.5ex>[r] & \ar@<.5ex>[l] \bullet \ar@{.}@(dr,dl) \ar@<.5ex>[r] & \ar@<.5ex>[l] \bullet \ar@{.}@(dr,dl) & \bullet \ar@{.}@(dr,dl) \ar@<.5ex>[r] & \ar@<.5ex>[l] \bullet \ar@{.}@(dr,dl) \ar@<.5ex>[r] & \ar@<.5ex>[l] \bullet \ar@{.}@(dr,dl) \ar@<.5ex>[r] & \ar@<.5ex>[l] \cdots
}
$}
\end{picture}
\item
When $f=(D^{2n}_\infty)$:

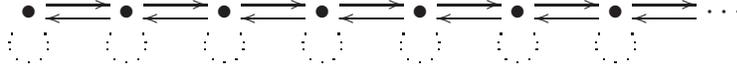
\begin{picture}(400,40)
\put (54,25){$
\xymatrix{
\bullet \ar@{.}@(dr,dl) \ar@<.5ex>[r] & \ar@<.5ex>[l] \bullet \ar@{.}@(dr,dl) \ar@<.5ex>[r] & \ar@<.5ex>[l] \bullet \ar@{.}@(dr,dl) \ar@<.5ex>[r] & \ar@<.5ex>[l] \bullet \ar@{.}@(dr,dl) \ar@<.5ex>[r] & \ar@<.5ex>[l] \bullet \ar@{.}@(dr,dl) \ar@<.5ex>[r] & \ar@<.5ex>[l] \bullet \ar@{.}@(dr,dl) \ar@<.5ex>[r] & \ar@<.5ex>[l] \bullet \ar@{.}@(dr,dl) \ar@<.5ex>[r] & \ar@<.5ex>[l] \cdots
}
$}
\end{picture}
\end{enumerate}
\end{theorem}

Consequently, combining our Theorem \ref{thm0} with this result due to Schreyer, we get a new understanding of the structure of the category of MCM modules over hypersurfaces of countable CM representation type.

We give two applications of our Theorem \ref{thm0}.
Using Theorem \ref{thm0}, we are able to calculate the dimension of the triangulated category $\underline{\CM}(R)$ in the sense of Rouquier \cite{R}, and the Grothendieck groups of $\CM(R)$ and $\underline{\CM}(R)$.

\begin{corollary}\label{cor3}
With the notation of Theorem \ref{thm0} the following hold.
\begin{enumerate}[\rm (1)]
\item
The triangulated category $\underline{\CM}(R)$ has dimension $1$.
\item
For $m\ge1$ one has:
\begin{align*}
K_0(\CM(R)) & =\langle[R],[X_R]\rangle\cong
\begin{cases}
\Z & \text{if }f=(A^1_\infty),\\
\Z^2 & \text{if }f = (A^{2m}_{\infty})\text{ or }f=(D^{2m-1}_{\infty}),\\
\Z \oplus \Z/2\Z & \text{if }f = (A^{2m+1}_{\infty})\text{ or } f = (D^{2m}_{\infty}),
\end{cases}
\\
K_0(\underline{\CM}(R)) & =\langle[\underline{X_R}]\rangle\cong
\begin{cases}
\Z & \text{if }f = (A_\infty^{2m})\text{ or }f=(D^{2m-1}_{\infty}),\\
\Z/2\Z & \text{if }f = (A^{2m-1}_{\infty})\text{ or }f = (D^{2m}_{\infty}).
\end{cases}
\end{align*}
\end{enumerate}
\end{corollary}

\begin{ac}
This research was conducted in large part while the first and second authors visited Shinshu University in August 2009, greatly supported by 2009 Discretionary Expense of the President of Shinshu University.
The authors are indebted to Yuji Yoshino for his useful and helpful suggestions.
The authors would like to thank Takuma Aihara, Igor Burban and Osamu Iyama very much for their valuable comments.
\end{ac}

\section{One and two dimensional cases}

In this section, we prove Theorem \ref{thm0} in the cases where the ring $R$ is of dimension $1$ and $2$.
The following proposition includes Theorem \ref{thm0} in the $1$-dimensional case.

\begin{proposition}\label{prop1}
Let $S$ be $2$-dimensional regular local ring and let $x_0,x_1$ be a regular system of parameters of $S$. 
\begin{enumerate}[\rm (1)]
\item
Let $R = S/(x_0^2)$.
Then $R/(x_0)\cong\Omega(R/(x_0))$, and the following statements hold.
\begin{enumerate}[\rm (a)]
\item
For every indecomposable MCM $R$-module $M$, there is an exact sequence $0 \to L \to M \to N \to 0$ of $R$-modules with $L,N \in \{ 0, R/(x_0) \}$. 
\item
One has $\mathcal{M}(R) = \{ R/(x_0) \}$. 
\item
One has $\NF(R/(x_0)) = \{ (x_0) , (x_0,x_1) \}$. 
\end{enumerate}
\item
Let $R = S/(x_0^2x_1)$.
Then $R/(x_0x_1)\cong\Omega(R/(x_0))$, and the following statements hold.
\begin{enumerate}[\rm (a)]
\item
For every indecomposable MCM $R$-module $M$, there is an exact sequence $0 \to L \to M \oplus R^{n} \to N \to 0$ with $L,N \in \{ 0, R/(x_0), R/(x_0x_1) \}$ and $n=0,1$. 
\item
One has $\mathcal{M}(R) = \{ R/(x_0), R/(x_0x_1) \}$. 
\item
One has $\NF(R/(x_0)) = \{ (x_0) , (x_0,x_1) \} = \NF(R/(x_0x_1))$
\end{enumerate}
\end{enumerate}
\end{proposition}

\begin{proof}
(1) By \cite[(4.1)]{BGS}, all the nonisomorphic indecomposable MCM $R$-modules are $R$, $R/(x_0)$ and $\Coker\,\varphi_n\ (n=1,2,\dots)$, where $\varphi_n = 
\begin{pmatrix} 
x_0&x_1^n\\
0&-x_0
\end{pmatrix}$.
We have the following short exact sequence of complexes.
$$
\begin{CD}
@. 0 @. 0 @. 0 \\
@. @VVV @VVV @VVV \\
(\cdots @>{x_0}>> R @>{x_0}>> R @>{x_0}>> R @>>> 0) \\
@. @V{
\left(
\begin{smallmatrix}
1 \\
0
\end{smallmatrix}
\right)
}VV @V{
\left(
\begin{smallmatrix}
1 \\
0
\end{smallmatrix}
\right)
}VV @V{
\left(
\begin{smallmatrix}
1 \\
0
\end{smallmatrix}
\right)
}VV \\
(\cdots @>{\varphi_n}>> R^2 @>{\varphi_n}>> R^2 @>{\varphi_n}>> R^2 @>>> 0) \\
@. @V{
\left(
\begin{smallmatrix}
0 & 1 
\end{smallmatrix}
\right)
}VV @V{
\left(
\begin{smallmatrix}
0 & 1 
\end{smallmatrix}
\right)
}VV @V{
\left(
\begin{smallmatrix}
0 & 1 
\end{smallmatrix}
\right)
}VV \\
(\cdots @>{-x_0}>> R @>{-x_0}>> R @>{-x_0}>> R @>>> 0) \\
@. @VVV @VVV @VVV \\
@. 0 @. 0 @. 0
\end{CD}
$$
Taking the long exact sequence of the homology modules, we get an exact sequence $0 \to R/(x_0) \to \Coker\,\varphi_n \to R/(x_0) \to 0$.
Also, the first row makes an exact sequence $0 \to R/(x_0) \to R \to R/(x_0) \to 0$.
We have $\NF(R/(x_0)) = \{ \p , \m \}=\Spec\,R$, where $\p=(x_0)$ and $\m=(x_0,x_1)$  (cf. \cite[(1.15(4))]{stcm}).
In particular, $R/(x_0)$ belongs to $\mathcal{M}(R)$. 
There is an equality $\begin{pmatrix} 
-\frac{x_0}{x_1^n}&-1\\
\frac{1}{x_1^n}&0
\end{pmatrix}
\begin{pmatrix} 
x_0&x_1^n\\
0&-x_0
\end{pmatrix}
\begin{pmatrix} 
1&0\\
-\frac{x_0}{x_1^n}&1
\end{pmatrix}
=
\begin{pmatrix} 
0&0\\
0&1
\end{pmatrix}$ of matrices over $R_{\p}$.
Hence $(\Coker\,\varphi_n)_{\p}\cong R_{\p}$ and $\Coker\,\varphi_n\notin\mathcal{M}(R)$. 
Therefore $\mathcal{M}(R) = \{ R/(x_0) \}$. 

(2) By \cite[(4.2)]{BGS}, all the nonisomorphic indecomposable MCM $R$-modules are $R$, $R/(x_0)$, $R/(x_0x_1)$, $R/(x_0^2)$, $R/(x_1)$, $\Coker\,\varphi^{+}_n$, $\Coker\,\varphi^{-}_n$, $\Coker\,\psi^{+}_n$ and $\Coker\,\psi^{-}_n\ (n=1,2,\dots)$, where $\varphi^{+}_n = 
\begin{pmatrix} 
x_0&x_1^n\\
0&-x_0
\end{pmatrix},\ \varphi^{-}_n = 
\begin{pmatrix} 
x_0x_1&x_1^{n+1}\\
0&-x_0x_1
\end{pmatrix},\ \psi^{+}_n = 
\begin{pmatrix} 
x_0x_1&x_1^n\\
0&-x_0
\end{pmatrix}$ and $\psi^{-}_n = 
\begin{pmatrix} 
x_0&x_1^n\\
0&-x_0x_1
\end{pmatrix}$.
Setting $x_1^0=1$, for $n\ge 0$ we have commutative diagrams
$$
\begin{CD}
0 @. 0 @. 0 @. 0 \\
@VVV @VVV @VVV @VVV \\
R @>{x_0}>> R @>{x_0x_1}>> R @>{x_0}>> R \\
@V{
\left(
\begin{smallmatrix}
1 \\
0
\end{smallmatrix}
\right)
}VV @V{
\left(
\begin{smallmatrix}
1 \\
0
\end{smallmatrix}
\right)
}VV @V{
\left(
\begin{smallmatrix}
1 \\
0
\end{smallmatrix}
\right)
}VV @V{
\left(
\begin{smallmatrix}
1 \\
0
\end{smallmatrix}
\right)
}VV \\
R^2 @>{\varphi_n^+}>> R^2 @>{\varphi_n^-}>> R^2 @>{\varphi_n^+}>> R^2 \\
@V{
\left(
\begin{smallmatrix}
0 & 1 
\end{smallmatrix}
\right)
}VV @V{
\left(
\begin{smallmatrix}
0 & 1 
\end{smallmatrix}
\right)
}VV @V{
\left(
\begin{smallmatrix}
0 & 1 
\end{smallmatrix}
\right)
}VV @V{
\left(
\begin{smallmatrix}
0 & 1 
\end{smallmatrix}
\right)
}VV \\
R @>{-x_0}>> R @>{-x_0x_1}>> R @>{-x_0}>> R \\
@VVV @VVV @VVV @VVV \\
0 @. 0 @. 0 @. 0
\end{CD}
\qquad\qquad
\begin{CD}
0 @. 0 @. 0 @. 0 \\
@VVV @VVV @VVV @VVV \\
R @>{x_0x_1}>> R @>{x_0}>> R @>{x_0x_1}>> R \\
@V{
\left(
\begin{smallmatrix}
1 \\
0
\end{smallmatrix}
\right)
}VV @V{
\left(
\begin{smallmatrix}
1 \\
0
\end{smallmatrix}
\right)
}VV @V{
\left(
\begin{smallmatrix}
1 \\
0
\end{smallmatrix}
\right)
}VV @V{
\left(
\begin{smallmatrix}
1 \\
0
\end{smallmatrix}
\right)
}VV \\
R^2 @>{\psi_n^+}>> R^2 @>{\psi_n^-}>> R^2 @>{\psi_n^+}>> R^2 \\
@V{
\left(
\begin{smallmatrix}
0 & 1 
\end{smallmatrix}
\right)
}VV @V{
\left(
\begin{smallmatrix}
0 & 1 
\end{smallmatrix}
\right)
}VV @V{
\left(
\begin{smallmatrix}
0 & 1 
\end{smallmatrix}
\right)
}VV @V{
\left(
\begin{smallmatrix}
0 & 1 
\end{smallmatrix}
\right)
}VV \\
R @>{-x_0}>> R @>{-x_0x_1}>> R @>{-x_0}>> R \\
@VVV @VVV @VVV @VVV \\
0 @. 0 @. 0 @. 0
\end{CD}
$$
with exact rows and columns.
Also we have equalities of matrices over $R$:
\begin{align*}
\begin{pmatrix}
x_0 & 1 \\
1 & 0
\end{pmatrix}
\begin{pmatrix}
x_0 & 1 \\
0 & -x_0
\end{pmatrix}
\begin{pmatrix}
1 & 0 \\
-x_0 & 1
\end{pmatrix}
& =
\begin{pmatrix}
x_0^2 & 0 \\
0 & 1
\end{pmatrix}
,\\
\begin{pmatrix}
1 & 0 \\
x_0 & 1
\end{pmatrix}
\begin{pmatrix}
x_0x_1 & x_1 \\
0 & -x_0x_1
\end{pmatrix}
\begin{pmatrix}
0 & 1 \\
1 & -x_0
\end{pmatrix}
& =
\begin{pmatrix}
x_1 & 0 \\
0 & 0
\end{pmatrix}
.
\end{align*}
Hence there are exact sequences
\begin{align*}
& 0 \to R/(x_0x_1) \to R \to R/(x_0) \to 0, \\
& 0 \to R/(x_0) \to R/(x_0^2) \to R/(x_0) \to 0, \\
& 0 \to R/(x_0x_1) \to R/(x_1) \oplus R \to R/(x_0x_1) \to 0, \\
& 0 \to R/(x_0) \to \Coker\,\varphi^{+}_n \to R/(x_0) \to 0, \\
& 0 \to R/(x_0x_1) \to \Coker\,\varphi^{-}_n \to R/(x_0x_1) \to 0, \\
& 0 \to R/(x_0x_1) \to \Coker\,\psi^{+}_n \to R/(x_0) \to 0, \\
& 0 \to R/(x_0) \to \Coker\,\psi^{-}_n \to R/(x_0x_1) \to 0.
\end{align*}
Put $\p=(x_0)$, $\q=(x_1)$ and $\m=(x_0,x_1)$.
Then we have $\Spec\,R=\{\p,\q,\m\}$, and easily see that the equalities $\NF(R/(x_0)) = \NF(R/(x_0x_1)) = \{ \p , \m \}$ and $\NF(R/(x_0^2)) = \NF(R/(x_1)) = \{ \m \}$ hold.
Therefore $R/(x_0), R/(x_0x_1) \in \mathcal{M}(R)$ and $R/(x_0^2), R/(x_1) \notin \mathcal{M}(R)$. 
Since $R_{\q}$ is a field, all $R_\q$-modules are free. 
There are equalities of matrices whose entries are in $R_{\p}$:
\begin{align*}
\begin{pmatrix} 
\frac{x_0}{x_1^n}&1\\
\frac{1}{x_1^n}&0
\end{pmatrix}
\begin{pmatrix} 
x_0&x_1^n\\
0&-x_0
\end{pmatrix}
\begin{pmatrix} 
1&0\\
-\frac{x_0}{x_1^n}&1
\end{pmatrix}
& =
\begin{pmatrix} 
0&0\\
0&1
\end{pmatrix},
\\
\begin{pmatrix} 
\frac{x_0}{x_1^n}&1\\
\frac{1}{x_1^n}&0
\end{pmatrix}
\begin{pmatrix} 
x_0x_1&x_1^n\\
0&-x_0
\end{pmatrix}
\begin{pmatrix} 
1&0\\
-\frac{x_0x_1}{x_1^n}&1
\end{pmatrix}
& =
\begin{pmatrix} 
0&0\\
0&1
\end{pmatrix}.
\end{align*}
Hence $(\Coker\,\varphi^{+}_n)_{\p} \cong R_{\p} \cong (\Coker\,\psi^{+}_n)_{\p}$.
Note that $\Coker\,\varphi^{-}_n$ and $\Coker\,\psi^{-}_n$ are the syzygies of $\Coker\,\varphi^{+}_n$ and $\Coker\,\psi^{+}_n$, respectively.
Thus $\Coker\,\varphi^{+}_n$, $\Coker\,\varphi^{-}_n$, $\Coker\,\psi^{+}_n$ and $\Coker\,\psi^{-}_n$ are not in $\mathcal{M}(R)$.
Consequently, we have $\mathcal{M}(R)$ = $\{ R/(x_0), R/(x_0x_1) \}$.
\end{proof}

Next, let us consider the case where the base ring has dimension $2$.

\begin{proposition}\label{prop2}
Let $k$ be an algebraically closed field. 
\begin{enumerate}[\rm (1)]
\item
Let $R = k[[x_0,x_1,x_2]]/(x_0x_2)$.
Then $R/(x_2)\cong\Omega(R/(x_0))$, and the following hold.
\begin{enumerate}[\rm (a)]
\item
For any indecomposable MCM $R$-module $M$, there is an exact sequence $0 \to L \to M \to N \to 0$ with $L,N \in \{ 0, R/(x_0), R/(x_2) \}$.
\item
One has $\mathcal{M}(R) = \{ R/(x_0), R/(x_2) \}$. 
\item
One has $\NF(R/(x_0)) = \{ (x_0,x_2) , (x_0,x_1,x_2) \} = \NF(R/(x_2))$.
\end{enumerate}
\item
Let $R = k[[x_0,x_1,x_2]]/(x_0^2x_1-x_2^2)$.
Then $(x_0,x_2)\cong\Omega(x_0,x_2)$, and the following hold.
\begin{enumerate}[\rm (a)]
\item
For any indecomposable MCM $R$-module $M$, there is an exact sequence $0 \to L \to M \oplus R^{n} \to N \to 0$ with $L,N \in \{ 0, (x_0,x_2) \}$ and $n=0,1$.
\item
One has $\mathcal{M}(R) = \{ (x_0,x_2) \}$. 
\item
One has $\NF((x_0,x_2)) = \{ (x_0,x_2) , (x_0,x_1,x_2) \}$.
\end{enumerate}
\end{enumerate}
\end{proposition}

\begin{proof}
(1) By \cite[(5.3)]{BD}, all indecomposable MCM $R$-modules are $R$, $R/(x_0)$, $R/(x_2)$, $\Coker\,\varphi^{+}_n$ and $\Coker\,\varphi^{-}_n\ (n=1,2,\dots)$, where $\varphi^{+}_n = 
\begin{pmatrix} 
x_2&x_1^n\\
0&x_0
\end{pmatrix},\ \varphi^{-}_n = 
\begin{pmatrix} 
x_0&-x_1^n\\
0&x_2
\end{pmatrix}$.
We have a commutative diagram with exact rows and columns:
$$
\begin{CD}
0 @. 0 @. 0 @. 0 \\
@VVV @VVV @VVV @VVV \\
R @>{x_2}>> R @>{x_0}>> R @>{x_2}>> R \\
@V{
\left(
\begin{smallmatrix}
1 \\
0
\end{smallmatrix}
\right)
}VV @V{
\left(
\begin{smallmatrix}
1 \\
0
\end{smallmatrix}
\right)
}VV @V{
\left(
\begin{smallmatrix}
1 \\
0
\end{smallmatrix}
\right)
}VV @V{
\left(
\begin{smallmatrix}
1 \\
0
\end{smallmatrix}
\right)
}VV \\
R^2 @>{\varphi_n^+}>> R^2 @>{\varphi_n^-}>> R^2 @>{\varphi_n^+}>> R^2 \\
@V{
\left(
\begin{smallmatrix}
0 & 1 
\end{smallmatrix}
\right)
}VV @V{
\left(
\begin{smallmatrix}
0 & 1 
\end{smallmatrix}
\right)
}VV @V{
\left(
\begin{smallmatrix}
0 & 1 
\end{smallmatrix}
\right)
}VV @V{
\left(
\begin{smallmatrix}
0 & 1 
\end{smallmatrix}
\right)
}VV \\
R @>{x_0}>> R @>{x_2}>> R @>{x_0}>> R \\
@VVV @VVV @VVV @VVV \\
0 @. 0 @. 0 @. 0
\end{CD}
$$
There are exact sequences $0 \to R/(x_2) \to R \to R/(x_0) \to 0$, $0 \to R/(x_2) \to \Coker\,\varphi^{+}_n \to R/(x_0) \to 0$ and $0 \to R/(x_0) \to \Coker\,\varphi^{-}_n \to R/(x_2) \to 0$.
We have $\NF(R/(x_0)) = \{ (x_0,x_2), (x_0,x_1,x_2) \}=\NF(R/(x_2))$, which implies $R/(x_0),R/(x_2)\in\mathcal{M}(R)$. 
Let $\p$ be any nonmaximal prime ideal.
Then one of $x_0,x_1,x_2$ is not in $\p$, and we have:
\begin{align*}
\begin{pmatrix} 
x_0&-x_1^n\\
0&\frac{1}{x_0}
\end{pmatrix}
\begin{pmatrix} 
x_2&x_1^n\\
0&x_0
\end{pmatrix}
=
\begin{pmatrix} 
0&0\\
0&1
\end{pmatrix}
& \quad\text{if } x_0 \notin \p,\\
\begin{pmatrix} 
\frac{x_0}{x_1^n}&-1\\
1&0
\end{pmatrix}
\begin{pmatrix} 
x_2&x_1^n\\
0&x_0
\end{pmatrix}
\begin{pmatrix} 
-x_1^n&x_0\\
x_2&\frac{1}{x_1^n}
\end{pmatrix}
=
\begin{pmatrix} 
0&0\\
0&1
\end{pmatrix}
& \quad\text{if } x_1 \notin \p,\text{ and}\\
\begin{pmatrix} 
0&1\\
1&0
\end{pmatrix}
\begin{pmatrix} 
x_2&x_1^n\\
0&x_0
\end{pmatrix}
\begin{pmatrix} 
-x_1^n&\frac{1}{x_2}\\
x_2&0
\end{pmatrix}
=
\begin{pmatrix} 
0&0\\
0&1
\end{pmatrix}
& \quad\text{if } x_2 \notin \p
\end{align*}
over $R_\p$.
In each case $(\Coker\,\varphi^{+}_n)_{\p}$ is isomorphic to $R_{\p}$.
Since $\Coker\,\varphi^{-}_n$ is the syzygy of $\Coker\,\varphi^{+}_n$, we see that $\Coker\,\varphi^{+}_n$ and $\Coker\,\varphi^{-}_n$ are not in $\mathcal{M}(R)$. 
Thus $\mathcal{M}(R) = \{ R/(x_0), R/(x_2) \}$ holds.

(2) By \cite[(5.7)]{BD} the following three assertions hold.
\begin{enumerate}[(i)]
\item
All indecomposable MCM $R$-modules are $R$, $\Coker\,\alpha^{+}$, $\Coker\,\alpha^{-}$, $\Coker\,\beta^{+}$, $\Coker\,\beta^{-}$, $\Coker\,\varphi^{+}_n$, $\Coker\,\varphi^{-}_n$, $\Coker\,\psi^{+}_n$ and $\Coker\,\psi^{-}_n\ (n=1,2,\dots)$, where $\alpha^{+} = \begin{pmatrix} 
x_2&x_0x_1\\
x_0&x_2
\end{pmatrix},\ \alpha^{-} = \begin{pmatrix} 
-x_2&x_0x_1\\
x_0&-x_2
\end{pmatrix}, \beta^{+} = \begin{pmatrix} 
x_0^2&x_2\\
x_2&x_1
\end{pmatrix},\ \beta^{-} = \begin{pmatrix} 
x_1&-x_2\\
-x_2&x_0^2
\end{pmatrix}$ and
\begin{align*}
& \varphi^{+}_n = 
\begin{pmatrix} 
x_2&x_0x_1&0&-x_1^{n+1}\\
x_0&x_2&x_1^n&0\\
0&0&x_2&x_0x_1\\
0&0&x_0&x_2
\end{pmatrix},\ 
\varphi^{-}_n = 
\begin{pmatrix} 
-x_2&x_0x_1&0&-x_1^{n+1}\\
x_0&-x_2&x_1^n&0\\
0&0&-x_2&x_0x_1\\
0&0&x_0&-x_2
\end{pmatrix},\\
& \psi^{+}_n = 
\begin{pmatrix} 
x_2&x_0x_1&-x_1^n&0\\
x_0&x_2&0&x_1^n\\
0&0&x_2&x_0x_1\\
0&0&x_0&x_2
\end{pmatrix},\ 
\psi^{-}_n = 
\begin{pmatrix} 
-x_2&x_0x_1&-x_1^n&0\\
x_0&-x_2&0&x_1^n\\
0&0&-x_2&x_0x_1\\
0&0&x_0&-x_2
\end{pmatrix}.
\end{align*}
\item
There are isomorphisms $\Coker\,\alpha^{+} \cong \Coker\,\alpha^{-}$, $\Coker\,\beta^{+} \cong \Coker\,\beta^{-}$, $\Coker\,\varphi^{+}_n \cong \Coker\,\varphi^{-}_n$ and $\Coker\,\psi^{+}_n \cong \Coker\,\psi^{-}_n$.
\item
The $R$-modules $\Coker\,\beta^{+},\ \Coker\,\varphi ^{+}_n,\ \Coker\,\psi^{+}_n$ are locally free on the punctured spectrum.
\end{enumerate}
We can easily check that the sequence $R^2\xrightarrow{\alpha^+}R^2\xrightarrow{(x_0,-x_2)}R\to R/(x_0,x_2)\to0$ is exact.
Hence $\Coker\,\alpha^{+}$ is isomorphic to the prime ideal $\p:=(x_0,x_2)$ of $R$.
For an ideal $I$ of $R$, denote by $\mathrm{V}(I)$ the set of prime ideals of $R$ containing $I$.
It is easy to see that $\NF(\p)$ is contained in $\mathrm{V}(\p)$, while $\p$ belongs to $\NF(\p)$ because the local ring $R_\p$ is not regular.
Thus we obtain $\NF(\p)=\mathrm{V}(\p)=\{\p,\m\}$, where $\m=(x_0,x_1,x_2)$ is the maximal ideal of $R$, and we have $\mathcal{M}(R)=\{\p\}$.
Setting $x_1^0=1$, for $n\ge 0$ we have commutative diagrams
$$
\begin{CD}
0 @. 0 @. 0 @. 0 \\
@VVV @VVV @VVV @VVV \\
R^2 @>{\alpha^+}>> R^2 @>{\alpha^-}>> R^2 @>{\alpha^+}>> R^2 \\
@V{
\left(
\begin{smallmatrix}
1 & 0 \\
0 & 1 \\
0 & 0 \\
0 & 0
\end{smallmatrix}
\right)
}VV @V{
\left(
\begin{smallmatrix}
1 & 0 \\
0 & 1 \\
0 & 0 \\
0 & 0
\end{smallmatrix}
\right)
}VV @V{
\left(
\begin{smallmatrix}
1 & 0 \\
0 & 1 \\
0 & 0 \\
0 & 0
\end{smallmatrix}
\right)
}VV @V{
\left(
\begin{smallmatrix}
1 & 0 \\
0 & 1 \\
0 & 0 \\
0 & 0
\end{smallmatrix}
\right)
}VV \\
R^4 @>{\varphi_n^+}>> R^4 @>{\varphi_n^-}>> R^4 @>{\varphi_n^+}>> R^4 \\
@V{
\left(
\begin{smallmatrix}
0 & 0 & 1 & 0 \\
0 & 0 & 0 & 1
\end{smallmatrix}
\right)
}VV @V{
\left(
\begin{smallmatrix}
0 & 0 & 1 & 0 \\
0 & 0 & 0 & 1
\end{smallmatrix}
\right)
}VV @V{
\left(
\begin{smallmatrix}
0 & 0 & 1 & 0 \\
0 & 0 & 0 & 1
\end{smallmatrix}
\right)
}VV @V{
\left(
\begin{smallmatrix}
0 & 0 & 1 & 0 \\
0 & 0 & 0 & 1
\end{smallmatrix}
\right)
}VV \\
R^2 @>{\alpha^+}>> R^2 @>{\alpha^-}>> R^2 @>{\alpha^+}>> R^2 \\
@VVV @VVV @VVV @VVV \\
0 @. 0 @. 0 @. 0
\end{CD}
\qquad\qquad
\begin{CD}
0 @. 0 @. 0 @. 0 \\
@VVV @VVV @VVV @VVV \\
R^2 @>{\alpha^+}>> R^2 @>{\alpha^-}>> R^2 @>{\alpha^+}>> R^2 \\
@V{
\left(
\begin{smallmatrix}
1 & 0 \\
0 & 1 \\
0 & 0 \\
0 & 0
\end{smallmatrix}
\right)
}VV @V{
\left(
\begin{smallmatrix}
1 & 0 \\
0 & 1 \\
0 & 0 \\
0 & 0
\end{smallmatrix}
\right)
}VV @V{
\left(
\begin{smallmatrix}
1 & 0 \\
0 & 1 \\
0 & 0 \\
0 & 0
\end{smallmatrix}
\right)
}VV @V{
\left(
\begin{smallmatrix}
1 & 0 \\
0 & 1 \\
0 & 0 \\
0 & 0
\end{smallmatrix}
\right)
}VV \\
R^4 @>{\psi_n^+}>> R^4 @>{\psi_n^-}>> R^4 @>{\psi_n^+}>> R^4 \\
@V{
\left(
\begin{smallmatrix}
0 & 0 & 1 & 0 \\
0 & 0 & 0 & 1
\end{smallmatrix}
\right)
}VV @V{
\left(
\begin{smallmatrix}
0 & 0 & 1 & 0 \\
0 & 0 & 0 & 1
\end{smallmatrix}
\right)
}VV @V{
\left(
\begin{smallmatrix}
0 & 0 & 1 & 0 \\
0 & 0 & 0 & 1
\end{smallmatrix}
\right)
}VV @V{
\left(
\begin{smallmatrix}
0 & 0 & 1 & 0 \\
0 & 0 & 0 & 1
\end{smallmatrix}
\right)
}VV \\
R^2 @>{\alpha^+}>> R^2 @>{\alpha^-}>> R^2 @>{\alpha^+}>> R^2 \\
@VVV @VVV @VVV @VVV \\
0 @. 0 @. 0 @. 0
\end{CD}
$$
with exact rows and columns.
There is an equality of matrices over $R$
$$
\begin{pmatrix}
0 & -x_0 & 0 & 1 \\
-1 & 0 & 0 & 0 \\
x_0 & -x_2 & 1 & 0 \\
0 & 1 & 0 & 0
\end{pmatrix}
\begin{pmatrix}
x_2 & x_0x_1 & 0 & -x_1 \\
x_0 & x_2 & 1 & 0 \\
0 & 0 & x_2 & x_0x_1 \\
0 & 0 & x_0 & x_2
\end{pmatrix}
\begin{pmatrix}
-1 & 0 & 0 & 0 \\
0 & 0 & 1 & 0 \\
x_0 & 0 & -x_2 & 1 \\
0 & 1 & x_0 & 0
\end{pmatrix}
=
\begin{pmatrix}
x_0^2 & x_2 & 0 & 0 \\
x_2 & x_1 & 0 & 0 \\
0 & 0 & 0 & 0 \\
0 & 0 & 0 & 1
\end{pmatrix},
$$
which gives an isomorphism $\Coker\,\varphi_0^+\cong\Coker\,\beta^+\oplus R$.
Consequently, we obtain exact sequences $0 \to \Coker\,\alpha^+ \to \Coker\,\beta^+\oplus R \to \Coker\,\alpha^+ \to 0$, $0 \to \Coker\,\alpha^{+} \to \Coker\,\varphi^{+}_n \to \Coker\,\alpha^{+} \to 0$ and $0 \to \Coker\,\alpha^{+} \to \Coker\,\psi^{+}_n \to \Coker\,\alpha^{+} \to 0$.
This completes the proof of the second assertion of the proposition.
\end{proof}

\section{General case}

In this section, we give a proof of Theorem \ref{thm0} in the general case.
Throughout this section, we assume that $k$ is an algebraically closed field of characteristic different from two.
First of all, let us recall the definition of a matrix factorization.

\begin{definition}
Let $S=k[[x_0,x_1,\dots,x_n]]$ and $0\ne f\in(x_0,x_1,\dots,x_n)S$.
A pair of square matrices $(A,B)$ with entries being in $S$ is called a {\it matrix factorization} of $f$ if it satisfies $AB=BA=fE$, where $E$ is an identity matrix.
\end{definition}

Let $\MF (f)$ be the category of matrix factorizations of $f$.
Eisenbud \cite[\S 6]{E} (see also \cite[Chap. 7]{Y}) proved that taking the cokernel induces a category equivalence $\MF (f)/\langle(1,f)\rangle\overset{\cong}{\longrightarrow}\CM (S/(f))$, where $\langle(1,f)\rangle$ denotes the ideal of the category $\MF (f)$ generated by $(1,f)$.
We identify $\CM (S/(f))$ with $\MF (f)/\langle(1,f)\rangle$.

The following lemma is called Kn\"orrer's periodicity (cf. \cite[(3.1)]{K}, \cite[Chap. 12]{Y}) and it will play a key role in this section.

\begin{lemma}\label{knorrer}
Let $S=k[[x_0,x_1,\dots,x_n]]$ and $T=k[[x_0,x_1,\dots,x_n,y,z]]$ be formal power series rings and $f \in (x_0,x_1,\dots,x_n)S$ a nonzero element.
The functor $\MF (f) \to \MF (f+yz)$ given by $(A,B)\mapsto\left(\left(
\begin{smallmatrix}
A & yE\\
zE & -B
\end{smallmatrix}
\right),\left(
\begin{smallmatrix}
B & yE\\
zE & -A
\end{smallmatrix}
\right)\right)$ induces a triangle equivalence between the stable categories $F: \underline{\CM} (S/(f))\overset{\cong}{\longrightarrow}\underline{\CM} (T/(f+yz))$.
\end{lemma}

Here we recall some basic properties of the stable category of MCM modules.
Let $R$ be a Henselian Gorenstein local ring.
For MCM $R$-modules $M$ and $N$, we have $M\cong N$ in $\underline{\CM}(R)$ if and only if $M\oplus R^m\cong N\oplus R^n$ in $\CM(R)$ for some $m,n\ge 0$.
Since $R$ is Henselian, for any object $M\in\underline{\CM}(R)$ there exists a unique object $M_0\in\CM(R)$ such that $M_0$ has no nonzero free summand and that $M_0\cong M$ in $\underline{\CM}(R)$.
If $M$ is an indecomposable object of $\underline{\CM}(R)$, then $M_0$ is a nonfree indecomposable MCM $R$-module.

\begin{proposition}\label{prop3}
Let $S,T,f$ and $F$ be as in Lemma \ref{knorrer}.
Put $R=T/(f+yz)$ and $R'=S/(f)$.
For a nonfree MCM $R'$-module $M$, the following statements hold.
\begin{enumerate}[\rm (1)]
\item
One has an inclusion $\NF_{R}(FM)\subseteq\mathrm{V}_R(y,z)$.
\item
One has a bijection $\Phi:\NF_{R}(FM) \rightarrow \NF_{R'}(M)$ which sends $\fP$ to $\fP/(y,z)$. 
\end{enumerate}
\end{proposition}

\begin{proof}
(1) Let $(A,B) \in \MF (f)$ be a matrix factorization corresponding to $M$.
Then $\Coker\left(
\begin{smallmatrix}
A & yE\\
zE & -B
\end{smallmatrix}
\right)$ is isomorphic to $FM$ up to free summand, and $FM$ is a nonfree MCM $R$-module.
Let $\fP\in\Spec\,R$ with $y\notin\fP$.
Then there is an equality $\begin{pmatrix} 
\frac{1}{y}B&E\\
\frac{1}{y}E&0
\end{pmatrix}
\begin{pmatrix} 
A & yE\\
zE & -B
\end{pmatrix}
\begin{pmatrix} 
yE&0\\
-A&E
\end{pmatrix}
=
\begin{pmatrix} 
0&0\\
0&E
\end{pmatrix}$ of matrices over $R_\fP$.
Hence $(FM)_{\fP}$ is a free $R_{\fP}$-module, which implies $\NF_{R}(FM)\subseteq\mathrm{V}_{R}(y)$.
Similarly, $\NF_{R}(FM)\subseteq\mathrm{V}_{R}(z)$ holds.

(2) The assignment $\fP\mapsto\fP/(y,z)$ makes a bijection $\mathrm{V}_R(y,z)\to\Spec\,R'$.
For $\fP \in \mathrm{V}_{R}(y,z)$, set $\p = \fP /(y,z) \in \Spec\, R'$.
It is enough to show that $\fP\in\NF_{R}(FM)$ if and only if $\p\in\NF_{R'}(M)$.
There are a ring isomorphism $R/(y,z) \cong R'$ and an $R'$-module isomorphism $FM/(y,z)FM \cong M \oplus \Omega_{R'} M$ up to free summand.
We have $R'_\p$-module isomorphisms $(FM)_{\fP}/(y,z)(FM)_{\fP}\cong \left( FM/(y,z)FM \right) _{\fP}\cong M_{\fP} \oplus (\Omega_{R'} M)_{\fP} \cong M_\p \oplus (\Omega_{R'} M)_\p \cong M_\p \oplus \Omega_{R'_\p} M_\p$ up to free $R'_\p$-summand.
Since $y,z$ is an $R$-regular sequence and $FM$ is a (nonzero) MCM $R$-module, the sequence $y,z$ is $FM$-regular.
Hence $(FM)_{\fP}$ is $R_{\fP}$-free if and only if $M_\p \oplus \Omega_{R'_\p} M_\p$ is $R'_\p$-free.
Therefore $\fP \in \NF_{R}(FM)$ if and only if $\p \in \NF_{R'}(M)$.
\end{proof}

Note that in general a nonfree finitely generated module $M$ over a commutative local ring $(R,\m)$ is locally free on the punctured spectrum of $R$ if and only if the equality $\NF_R(M)=\{\m\}$ holds.
Thus Proposition \ref{prop3} yields the corollary below.

\begin{corollary}\label{cor1}
With the notation of Proposition \ref{prop3}, $M$ is locally free on the punctured spectrum of $R'$ if and only if $FM$ is locally free on the punctured spectrum of $R$. 
\end{corollary}

Now we can prove Theorem \ref{thm0} in the general case.

\begin{tpf}
Let $R = k[[x_0,x_1,x_2,\ldots ,x_d]]/(f)$, where $f$ is either $({A}^d_{\infty})$ or $({D}^d_{\infty})$.
We induce on $d=\dim R$.
We may assume $d \geq 3$ thanks to Propositions \ref{prop1} and \ref{prop2}.
Put $R' = k[[x_0,x_1,x_2,\ldots ,x_{d-2}]]/(f')$ such that $f=f'+x_{d-1}^2+x_d^2$.
Then $f'$ is either $({A}^{d-2}_{\infty})$ or $({D}^{d-2}_{\infty})$.
As the characteristic of $k$ is not $2$, the ring $R$ is isomorphic to $k[[x_0,x_1,x_2,\ldots ,x_{d-2},y,z]]/(f'+yz)$, where $y,z$ are indeterminates over $k[[x_0,x_1,x_2,\dots,x_{d-2}]]$.
Let $F:\underline{\CM} (R') \to \underline{\CM} (R)$ be the equivalence in Lemma \ref{knorrer}.

(1) There is a nonfree indecomposable MCM $R'$-module $X_{R'}$ with $\mathcal{M}(R') = \{ X_{R'}, \Omega_{R'}(X_{R'})\}$ and $\NF_{R'}(X_{R'}) = \{ (x_0, x_2, \dots , x_{d-2})R', (x_0, x_1, x_2,\dots , x_{d-2})R'\} = \NF_{R'}(\Omega _{R'}(X_{R'}))$ by the induction hypothesis. 
Let $X_R$ be the nonfree direct summand of $FX_{R'}$.
Recall that the shift functor on the triangulated category $\underline{\CM} (R)$ (respectively, $\underline{\CM} (R')$) is the cosyzygy functor $\Omega^{-1}_R$ (respectively, $\Omega^{-1}_{R'}$).
Hence $F$ commutes with the syzygy functor, and we have $F(\Omega _{R'}(X_{R'})) \cong \Omega _R (FX_{R'}) \cong \Omega _R (X_R)$ in $\underline{\CM} (R)$.
It follows from this that $\Omega _R(X_R)$ is the nonfree direct summand of $F(\Omega _{R'}(X_{R'}))$.
By Proposition \ref{prop3} and Corollary \ref{cor1}, we see that $\mathcal{M}(R) = \{ X_R, \Omega_R (X_R) \}$ and that $\NF_R(X_{R}) = \{ (x_0, x_2, \dots , x_{d})R, (x_0, x_1, x_2,\dots , x_{d})R\} = \NF_R(\Omega _{R}(X_{R}))$.

(2) Let $M\in\mathcal{P}(R)$ be an indecomposable $R$-module.
Then there exists $M'\in\CM(R')$ such that $FM'$ is isomorphic to $M$ up to free summand.
Since the $R$-module $M$ is indecomposable, $M'$ can be chosen as an indecomposable $R'$-module.
Corollary \ref{cor1} implies that $M'$ is in $\mathcal{P}(R')$.
By induction hypothesis, there is an exact sequence $0 \to L' \to M'\oplus R'^n \to N' \to 0$, where $L', N' \in \mathcal{M}(R')$ and $n\ge 0$. 
Then we obtain an exact triangle $L' \to M' \to N' \to \Omega _{R'}^{-1} L'$ in $\underline{\CM} (R')$, which gives an exact triangle $FL' \to M \to FN' \to \Omega _R^{-1}FL'$ in $\underline{\CM} (R)$.
Let $L$ and $N$ be the nonfree direct summands of $FL'$ and $FN'$, respectively.
Then we have an exact triangle $L\to M\to N \to \Omega_R^{-1}L$, which gives a short exact sequence $0 \to L \to M \oplus R^m \to N \to 0$ of $R$-modules.
Since $L'$ and $N'$ belong to $\mathcal{M}(R')$, the modules $L$ and $N$ are in $\mathcal{M}(R)$ by Corollary \ref{cor1}.
\qed
\end{tpf}

\section{Applications}

In this section, we give some applications of Theorem \ref{thm0}.
First, we calculate the dimension of the triangulated category $\underline{\CM}(R)$.
For the definition of the dimension of a triangulated category, see \cite[(3.2)]{R}.

\begin{proposition}
The dimension of $\underline{\CM}(R)$ is equal to $1$.
\end{proposition}

\begin{proof}
Theorem \ref{thm0}(2) especially says that the dimension of $\underline{\CM}(R)$ is at most $1$.
Note that our hypersurface $R$ is not of finite CM representation type and that every MCM $R$-module is isomorphic to its second syzygy up to free summand.
Hence the dimension of $\underline{\CM}(R)$ is nonzero.
Now the conclusion follows.
\end{proof}

Next, let us consider the Grothendieck group $K_0(\CM(R))$ of $\CM(R)$.
Applying Proposition \ref{prop1} and \ref{prop2}, we can calculate the Grothendieck group $K_0(\CM(R))$ of $\CM(R)$ for the hypersurfaces $R$ of types $(A_\infty^1)$, $(D_\infty^1)$, $(A_\infty^2)$ and $(D_\infty^2)$.

\begin{proposition}\label{1755}
With the notation of Theorem \ref{thm0}, we have
$$
K_0(\CM(R))\cong
\begin{cases}
\Z & \text{if }f = (A^1_{\infty}),\\
\Z^2 & \text{if }f = (D^1_{\infty})\text{ or }f = (A^2_{\infty}),\\
\Z \oplus \Z/2\Z & \text{if }f = (D^2_{\infty}).
\end{cases}
$$
\end{proposition}

\begin{proof}
Let $R$ be as in Theorem \ref{thm0} and $d=1,2$.
By Propositions \ref{prop1} and \ref{prop2}, there is an epimorphism $\Z^2 \to K_0(\CM(R))$.
Indeed, sending the canonical basis $\{\binom{1}{0},\binom{0}{1}\}$ of $\Z^2$ to
\begin{align*}
\{[R],[R/(x_0)]\} \quad & \text{if }f=(A_\infty^1),(D_\infty^1)\text{ or }(A_\infty^2),\\
\{[R],[(x_0,x_2)]\} \quad & \text{if }f=(D_\infty^2)
\end{align*}
makes such a surjection.
We get an exact sequence $0 \to \Z^{2-r} \to \Z^2 \to K_0(\CM(R)) \to 0$, where $r$ is the rank of $K_0(\CM(R))$.
Let $Q$ be the total quotient ring of $R$.
For a commutative ring $A$, denote by $\mod\,A$ the category of finitely generated $A$-modules. 
Define homomorphisms $a:K_0(\CM(R))\to K_0(\mod\,R)$ and $b:K_0(\mod\,R) \to K_0(\mod\,Q)$ by $a([M])=[M]$ and $b([N])=[N\otimes_RQ]$ for $M\in\CM(R)$ and $N\in\mod\,R$.
For $N\in\mod\,R$, there is an exact sequence $0 \to X_n \to X_{n-1} \to \cdots \to X_0 \to N \to 0$ with $X_i\in\CM(R)$ for $0\le i\le n$.
(For instance, take a free resolution of $N$.)
Then in $K_0(\mod\,R)$ the equality $[N]=\sum_{i=0}^n(-1)^i[X_i]$ holds, which shows that $a$ is surjective.
Clearly, $b$ is also.
Taking the composition, we have a surjection $K_0(\CM(R))\to K_0(\mod\,Q)$.
As $Q$ is Artinian, every finitely generated $Q$-module has finite length, and $K_0(\mod\,Q)$ is the free $\Z$-module with basis $\{ [Q/\mathfrak M]\mid\mathfrak M\text{ is a maximal ideal of }Q\}$ (cf. \cite[(1.7)]{ARS}).

When $f=(A^1_{\infty})$, the ring $Q$ has a unique maximal ideal.
Hence $K_0(\mod\,Q)\cong\Z$.
Since there is an exact sequence $0\to R/(x_0)\to R\to R/(x_0)\to 0$, the equality $[R]=2[R/(x_0)]$ holds in $K_0(\CM(R))$.
Therefore we have $r=1$.
There is a commutative diagram
$$
\begin{CD}
0 @>>> \Z  @>{
\left(
\begin{smallmatrix}
1 \\
-2
\end{smallmatrix}
\right)}>>  \Z^2 @>>> K_0(\CM(R)) @>>> 0 \\
@. @| @V{\cong}V{
\left(
\begin{smallmatrix}
1 & 0 \\
2 & 1
\end{smallmatrix}
\right)
}V @VVV \\
0 @>>> \Z @>{
\left(
\begin{smallmatrix}
1 \\
0
\end{smallmatrix}
\right)}>> \Z^2 @>>> \Z @>>> 0
\end{CD}
$$
with exact rows.
Thus, the $\Z$-module $K_0(\CM(R))$ is isomorphic to $\Z$.

When either $f=(D^1_{\infty})$ or $f=(A^2_{\infty})$, the ring $Q$ has two maximal ideals, and we have $K_0(\mod\,Q)\cong\Z^2$. 
The $\Z$-module $K_0(\CM(R))$ is also isomorphic to $\Z^2$.

When $f=(D^2_{\infty})$, the ring $Q$ is a field.
Hence $K_0(\mod\,Q)\cong\Z$.
With the notation of the proof of Proposition \ref{prop2}(2), we have isomorphisms $\Coker\,\alpha^-\cong\Coker\,\alpha^+\cong(x_0,x_2)$ and an exact sequence $0\to\Coker\,\alpha^-\to R^2\to\Coker\,\alpha^+\to 0$.
Thus $K_0(\CM(R))$ has a relation $2[R]=2[(x_0,x_2)]$, and we see that $r=1$.
There is a commutative diagram
$$
\begin{CD}
0 @>>> \Z  @>{
\left(
\begin{smallmatrix}
2 \\
-2
\end{smallmatrix}
\right)}>>  \Z^2 @>>> K_0(\CM(R)) @>>> 0 \\
@. @| @V{\cong}V{
\left(
\begin{smallmatrix}
1 & 1 \\
0 & -1
\end{smallmatrix}
\right)
}V @VVV \\
0 @>>> \Z @>{
\left(
\begin{smallmatrix}
0 \\
2
\end{smallmatrix}
\right)}>> \Z^2 @>>> \Z\oplus\Z/2\Z @>>> 0
\end{CD}
$$
with exact rows.
Therefore $K_0(\CM(R))$ is isomorphic to $\Z \oplus \Z/2\Z$.
\end{proof}

Before moving to the case of higher dimension, we verify that the following relasionship exists between the Grothendieck group of $\CM(R)$ for a general Gorenstein complete local domain $R$ and that of $\underline{\CM}(R)$.

\begin{lemma}\label{k0}
Let $R$ be a complete Gorenstein local domain.
Then we have an exact sequence
$$
0 \to \langle [R] \rangle \xrightarrow{f} K_0(\CM(R)) \xrightarrow{g} K_0(\underline{\CM}(R)) \to 0
$$
of $\Z$-modules, where $f,g$ are natural maps.
Moreover, $\langle [R] \rangle \cong \Z$ holds.
\end{lemma}

\begin{proof}
It is trivial that $f$ is an injective map.
As an exact sequence $0 \to X \to Y \to Z \to 0$ of MCM $R$-modules induces an exact triangle $\underline{X} \to \underline{Y} \to \underline{Z} \to \underline{X}[1]$ in $\underline{\CM}(R)$, we have a well-defined map $g$.
Clearly, $g$ is surjective and $gf=0$.
Let $\underline{X} \to \underline{Y} \to \underline{Z} \to \underline{X}[1]$ be an exact triangle in $\underline{\CM}(R)$.
Then we have an exact sequence $0 \to X \to Y \oplus R^n \to Z \to 0$ of MCM $R$-modules, which gives an equality $[X] - [Y] + [Z] = n[R]$ in $K_0(\CM(R))$.
Thus we have the exact sequence in the lemma.
As to the last statement, we have a map $K_0(\CM(R))\to\Z$ given by $[M]\mapsto \rank_R\,M$, where $\rank_R\,M$ denotes the rank of $M$.
The restriction of this map to $\langle [R] \rangle$ is the inverse map of the natural surjection $\Z\to\langle [R] \rangle$.
\end{proof}

The Grothendieck group of $\underline{\CM}(R)$ for a hypersurface $R$ of countable CM representation type is described as follows.

\begin{proposition}\label{prop4}
With the notation of Theorem \ref{thm0}, for $m \geq 1$ we have
$$
K_0(\underline{\CM}(R))\cong
\begin{cases}
\Z/2\Z & \text{if }f = (A_{\infty}^{2m-1})\text{ or }f = (D_{\infty}^{2m}),\\
\Z & \text{if }f = (D_{\infty}^{2m-1})\text{ or }f = (A_{\infty}^{2m}).
\end{cases}
$$
\end{proposition}

\begin{proof}
By virtue of Kn\"{o}rrer's periodicity (Lemma \ref{knorrer}), we have only to deal with the cases $\dim R=1,2$.
Using Propositions \ref{prop1} and \ref{prop2}, we obtain the assertion.
\end{proof}

\begin{proposition}\label{prop5}
With the notation of Theorem \ref{thm0}, for $m\ge1$ we have
$$
K_0(\CM(R))\cong
\begin{cases}
\Z \oplus \Z/2\Z & \text{if }f = (A^{2m+1}_{\infty})\text{ or } f = (D^{2m+2}_{\infty}),\\
\Z^2 & \text{if }f = (D^{2m+1}_{\infty})\text{ or }f = (A^{2m+2}_{\infty}).
\end{cases}
$$
\end{proposition}

\begin{proof}
Let $R$ be as in Theorem \ref{thm0}.
As $m\ge1$, the ring $R$ is an integral domain.
Tensoring the quotient field $Q$ of $R$ induces a surjection $K_0(\CM(R))\to K_0(\mod\,Q)$ (this is nothing but the composition $ba$ with the notation in the proof of Proposition \ref{1755}), and $K_0(\mod\,Q)\cong\Z$ as $Q$ is a field.
Hence $K_0(\CM(R))$ has a direct summand isomorphic to $\Z$.
We see from Theorem \ref{thm0}(2) that $K_0(\CM(R))\cong\Z\oplus\Z/e\Z$ for some $e\ge0$.
Let $F:\underline{\CM}(R')\to\underline{\CM}(R)$ be the $m$-th power of the triangle equivalence given in Lemma \ref{knorrer}, where $R'$ is the corresponding hypersurface of dimension $1$ or $2$.

When $f$ is either $(D^{2m+1}_{\infty})$ or $(A^{2m+2}_{\infty})$, Proposition \ref{prop4} and Lemma \ref{k0} imply the $\Z$-isomorphism $K_0(\CM(R)) \cong \Z^2$.

Let $f=(A^{2m+1}_{\infty})$.
Then, since there is an exact sequence $0\to F((x_0),(x_0))\to R^{2^m}\to F((x_0),(x_0))\to 0$, the equality $2^m[R]=2[F((x_0),(x_0))]$ holds in $K_0(\CM(R))$.
This gives a commutative diagram
$$
\begin{CD}
0 @>>> \Z  @>{
\left(
\begin{smallmatrix}
2^m \\
-2
\end{smallmatrix}
\right)}>> \Z^2 @>>> K_0(\CM(R)) @>>> 0 \\
@. @| @V{\cong}V{
\left(
\begin{smallmatrix}
1 & 2^{m-1} \\
0 & -1
\end{smallmatrix}
\right)
}V @VVV \\
0 @>>> \Z @>{
\left(
\begin{smallmatrix}
0 \\
2
\end{smallmatrix}
\right)}>> \Z^2 @>>> \Z\oplus\Z/2\Z @>>> 0
\end{CD}
$$
with exact rows.
Thus, the $\Z$-module $K_0(\CM(R))$ is isomorphic to $\Z\oplus\Z/2\Z$.

Let $f=(D^{2m+2}_{\infty})$.
Then, with the notation of the proof of Proposition \ref{prop2}(2), we have an isomorphism $F(\alpha^-,\alpha^+) \cong F(\alpha^+,\alpha^-)$ and an exact sequence $0\to F(\alpha^-,\alpha^+)\to R^{2^{m+1}}\to F(\alpha^+,\alpha^-)\to 0$.
Hence $K_0(\CM(R))$ has a relation $2^{m+1}[R]=2[F(\alpha^+,\alpha^-)]$.
A commutative diagram
$$
\begin{CD}
0 @>>> \Z  @>{
\left(
\begin{smallmatrix}
2^{m+1} \\
-2
\end{smallmatrix}
\right)}>> \Z^2 @>>> K_0(\CM(R)) @>>> 0 \\
@. @| @V{\cong}V{
\left(
\begin{smallmatrix}
1 & 2^m \\
0 & -1
\end{smallmatrix}
\right)
}V @VVV \\
0 @>>> \Z @>{
\left(
\begin{smallmatrix}
0 \\
2
\end{smallmatrix}
\right)}>> \Z^2 @>>> \Z\oplus\Z/2\Z @>>> 0
\end{CD}
$$
with exact rows exists, which shows that $K_0(\CM(R))$ is isomorphic to $\Z \oplus \Z/2\Z$.
\end{proof}





\begin{thebibliography}{99}

\bibitem{ARS}
{\sc M. Auslander}; {\sc I. Reiten}; {\sc S. O. Smal\o}, {\it Representation Theory of Artin Algebra}, Cambridge Studies in Advanced Mathematics, 36, {Cambridge University Press, Cambridge}, 1995.

\bibitem{BGS}
{\sc R.-O. Buchweitz}; {\sc G.-M. Greuel}; {\sc F.-O. Schreyer}, Cohen-Macaulay modules on hypersurface singularities II, {\it Invent. math.} {\bf 88} (1987), 165--182.

\bibitem{BD}
{\sc I. Burban}; {\sc Y. Drozd}, Maximal Cohen-Macaulay modules over surface singularities, {\it Trends in representation theory of algebras and related topics, EMS Ser. Congr. Rep., Eur. Math. Soc., Zurich}, (2008), 101--166, \texttt{http://arXiv.org/abs/0803.0117}.

\bibitem{E}
{\sc D. Eisenbud}, Homological algebra on a complete intersection, with an application to group representations, {\it Trans. Amer. Math. Soc.} {\bf 260} (1980), 35--64.

\bibitem{GK}
{\sc G.-M. Greuel}; {\sc H. Kr\"oning}, Simple singularities in positive characteristic, {\it Math. Z.} {\bf 203} (1990), 339--354.

\bibitem{K}
{\sc H. Kn\"orrer}, Cohen-Macaulay modules on hypersurface singularities I, {\it Invent. math.} {\bf 88} (1987), 153--164.

\bibitem{R}
{\sc R. Rouquier}, Dimensions of triangulated categories, {\it J. K-Theory} {\bf 1} (2008), 193--256.

\bibitem{Sc}
{\sc F.-O. Schreyer}, Finite and countable CM-representation type, Singularities, representation of algebras, and vector bundles, {\it Springer Lecture Notes in Math.} {\bf 1273} (1987), 9--34.

\bibitem{res}
{\sc R. Takahashi}, Modules in resolving subcategories which are free on the punctured spectrum, {\it Pacific J. Math.} {\bf 241} (2009), no. 2, 347--367.

\bibitem{stcm}
{\sc R. Takahashi}, Classifying thick subcategories of the stable category of Cohen-Macaulay modules, {\it Adv. in Math.} {\bf 225} (2010), 2076--2116.

\bibitem{Y}
{\sc Y. Yoshino}, {\it Cohen-Macaulay modules over Cohen-Macaulay rings}, London mathematical Society Lecture Note Series, 146, {Cambridge University Press, Cambridge}, 1990.
\end{thebibliography}
\end{document}